\documentclass[11pt]{amsart}
\usepackage{mathptmx}
\usepackage{amscd,amssymb,amsmath,amsthm,amsfonts}
\usepackage{graphicx}

\newtheorem{theorem}[equation]{Theorem}
\newtheorem{prop}[equation]{Proposition}
\newtheorem{lemma}[equation]{Lemma}
\newtheorem{cor}[equation]{Corollary}

\theoremstyle{remark}
\newtheorem{remark}[equation]{Remark} 
\theoremstyle{definition}

\newtheorem{example}[equation]{Example}
\numberwithin{equation}{subsection}

\newtheorem{open}{Open Problem}

\DeclareMathAlphabet{\matheur}{U}{eur}{m}{n}
\DeclareMathAlphabet{\matheus}{U}{eus}{m}{n}
\DeclareMathAlphabet{\matheuf}{U}{euf}{m}{n}
\newcommand{\abs}[1]{\left\lvert#1\right\rvert}

\newcommand{\Ric}{\mathit{Ricci}}
\newcommand{\vol}{\mathit{Vol}}
\DeclareMathOperator{\dist}{dist}
\DeclareMathOperator{\inj}{inj}

\DeclareMathOperator{\card}{card}
\DeclareMathOperator{\mult}{mult}
\DeclareMathOperator{\rad}{rad}
\usepackage{hyperref}
\author{Asma Hassannezhad}
\address{Max-Planck Institute for Mathematics, Vivatsgasse 7, 53111 Bonn, Germany}
\email{\tt  hassannezhad@mpim-bonn.mpg.de }

\author{Gerasim Kokarev}
\address{Mathematisches Institut der Universit\"at M\"unchen, Theresienstr. 39, D-80333 M\"unchen, Germany}
\email{\tt Gerasim.Kokarev@math.lmu.de}
\curraddr{School of Mathematics, The University of Leeds, Leeds, LS2 9JT, UK}
\email{G.Kokarev@leeds.ac.uk}

\author{Iosif Polterovich}
\address{D\'epartement de math\'ematiques et de statistique, Universit\'e de Montr\'eal, CP 6128 succ Centre-Ville, Montr\'eal, QC H3C 3J7, Canada}
\email{iossif@dms.umontreal.ca}


\title{Eigenvalue inequalities on Riemannian manifolds with a lower Ricci curvature bound}

\begin{document}

\begin{abstract}
We  revisit classical eigenvalue inequalities due to  Buser, Cheng,  and Gromov  on closed Riemannian manifolds, and prove the versions of these results  for the Dirichlet and Neumann boundary value problems. Eigenvalue multiplicity bounds and related open problems are also discussed.

\end{abstract}

\maketitle
\section{Statements and discussion}
\label{results}
\subsection{Notation and preliminaries}
Let $(M,g)$ be a compact $n$-dimensional  Riemannian manifold with or without boundary, and let
$$
0=\lambda_0(g)<\lambda_1(g)\leqslant\lambda_2(g)\leqslant\ldots\leqslant\lambda_k(g)\leqslant\ldots
$$ 
be the eigenvalues of the Laplace-Beltrami operator on $M$. If the boundary $\partial M$ is non-empty we assume for now that the Neumann boundary conditions are imposed. Later we also consider the Dirichlet eigenvalue problem; its eigenvalues are denoted by $\nu_k(g)$.  Recall that the eigenvalue counting function $N_g(\lambda)$ is defined for any $\lambda> 0$ as the number of  Laplace eigenvalues, counted with multiplicity, that are strictly less than $\lambda$. By the celebrated Weyl's law the counting function satisfies the following asymptotics:
\begin{equation}
\label{Weyl}
N_g(\lambda)\sim\frac{\omega_n}{(2\pi)^n}\vol_g(M)\lambda^{n/2}\qquad\text{as}\quad\lambda\to+\infty,
\end{equation}
where $\vol_g(M)$ is the volume of $M$ and $\omega_n$ is the volume of a unit ball in the Euclidean space~$\mathbb R^n$, see~\cite{SV} for the refined asymptotics and other developments in the subject.  By $m_k(g)$ we denote the multiplicity of the $k$th eigenvalue $\lambda_k(g)$. Clearly, we have $m_k(g)=N_g(\lambda_k+0)-N_g(\lambda_k)$, and hence, $m_k(g)=o(\lambda_k^{n/2})$ as $k\to +\infty$.

The purpose of this paper is three-fold. First, we revisit classical lower bounds for Laplace eigenvalues on closed Riemannian manifolds due to Gromov and Buser, and give an alternative unified approach to these statements. It avoids delicate isoperimetric arguments used in the original proofs, and uses only the Neumann-Poincar\'e inequality and geometric estimates for the cardinality of certain coverings. The advantage of our argument is that it carries over directly to the boundary value problems for geodesically convex domains, and yields rather explicit eigenvalue bounds, which appear to be new. Next, we turn our attention to the eigenvalue upper bounds originally obtained by  Cheng and Buser on closed manifolds. Some of their versions for boundary value problems also appear to be missing in the literature, and we fill this gap by presenting such results. Finally, we discuss eigenvalue multiplicity bounds on Riemannian manifolds, showing, for example, that for geodesically convex compact domains in complete manifolds of non-negative Ricci curvature the multiplicities $m_k(g)$ of Neumann eigenvalues are bounded in terms of the dimension and the index $k$ only. We end Section~\ref{results} with a few related open problems. Section~\ref{prem} contains the necessary background material, and the proofs of lower and upper eigenvalue bounds appear in Sections~\ref{proofs:lower} and~\ref{proofs:upper}  respectively.

\subsection{Lower eigenvalue bounds: Gromov and Buser revisited}
\label{section:lower}
Let $M$ be a closed manifold of non-negative Ricci curvature. A classical result by Li and Yau~\cite{LY80} says that the first Laplace eigenvalue $\lambda_1(g)$ of $M$ satisfies the inequality $\lambda_1(g)\geqslant \pi^2/(4d^2)$, where $d$ is the diameter of $M$. Later it has been improved by Zhong and Yang~\cite{ZY84} to the estimate $\lambda_1(g)\geqslant \pi^2/d^2$. For more general closed manifolds the following inequalities for all Laplace eigenvalues hold.
\begin{theorem}
\label{gromov}
Let $(M,g)$ be a closed Riemannian manifold whose Ricci curvature satisfies the bound $\Ric\geqslant -(n-1)\kappa$, where $\kappa\geqslant 0$. Then there exist constants $C_i$, $i=1,\ldots,3$, depending on the dimension $n$ of $M$  only,  such that:
\begin{equation}
\label{leb:gro}
\lambda_k(g)\geqslant C_1^{1+d\sqrt{\kappa}}d^{-2}k^{2/n}\qquad\text{for any}\quad k\geqslant 1,
\end{equation}
and 
\begin{equation}
\label{leb:bu}
\lambda_k(g)\geqslant C_2\vol_g(M)^{-2/n}k^{2/n} \qquad\text{for any}\quad k\geqslant 3C_3\vol_g(M)\max\{\kappa^{n/2},\inj^{-n}\},
\end{equation}
where $d$ and $\inj$ are the diameter and the injectivity radius of $M$ respectively.
\end{theorem}
Inequality~\eqref{leb:gro} is due to Gromov~\cite[Appendix~C]{Gro}. Motivated by Weyl's law~\eqref{Weyl}, he also poses a question whether there is an asymptotically sharp lower bound in terms of volume. This question has been answered by Buser who proved inequality~\eqref{leb:bu}, which however has been stated in~\cite[Theorem~6.2]{Bu82} in a slightly different form.  Note  that the hypothesis on the index $k$ in~\eqref{leb:bu} is necessary: for any given integer $k$ no geometry-free lower bound for the renormalised eigenvalue $\lambda_k(g)\vol_g(M)^{2/n}$ can hold. Indeed, the standard examples of manifolds with long necks (the connected sums of the so-called Cheeger dumbbells) show that there are sequences of metrics whose $k$th renormalised eigenvalues converge to zero. Moreover, as the examples of long thin flat tori show, the appearance of the injectivity radius in the hypothesis on the index $k$ also can not be easily removed. A number of related eigenvalue bounds have also been obtained by Li and Yau~\cite{LY80}, and Donnelly and Li~\cite{DL82}.

Both arguments by Gromov and Buser use methods based on isoperimetric inequalities: in the former case it is the circle of ideas around Levy's isoperimetric inequality, and in the latter -- the estimate for the Cheeger constant. The lower eigenvalue bounds above can be  re-written in the form of upper bounds for the counting function $N_g(\lambda)$. In particular, Gromov's bound~\eqref{leb:gro} is equivalent to the inequality
\begin{equation}
\label{leb:gro1}
N_g(\lambda)\leqslant \max\{C_4^{1+d\sqrt{\kappa}}d^n{\lambda}^{n/2}, 1\}\qquad\text{for any}\quad \lambda \geqslant 0,
\end{equation}
and Buser's inequality~\eqref{leb:bu}
is a consequence of 
\begin{equation}
\label{improved}
\quad N_g(\lambda)\leqslant C_3\vol_g(M)({\lambda}^{n/2}+\kappa^{n/2}+\inj^{-n}) \qquad\text{for any}\quad \lambda \geqslant 0.
\end{equation}
In Section~\ref{proofs:lower} we give a rather direct argument for the inequalities~\eqref{leb:gro1} and~\eqref{improved} that allows to bound the values of the counting function via the cardinality of an appropriate covering by metric balls, and avoids using  isoperimetric inequalities as in~\cite{Bu82,Gro}. We also obtain versions of these inequalities for boundary value problems, which we discuss now.

Suppose that $(M,g)$ is a complete Riemannian manifold whose Ricci curvature is bounded below $\Ric\geqslant -(n-1)\kappa$, where $\kappa\geqslant 0$, and $\Omega\subset M$ is a compact domain with a Lipschitz boundary. 
Recall that a domain with a  smooth boundary is said to satisfy the {\it interior rolling $\delta$-ball condition} if for any $x\in\partial\Omega$ there exists a ball $B$ of radius $\delta$ contained in $\Omega$ that touches the boundary $\partial\Omega$ at the point $x$, that is $B\subset\Omega$ and  $B\cap\partial\Omega=\{x\}$.
We define the {\it maximal radius} $\rad(\Omega)$ of an interior rolling ball as the supremum of $\delta>0$ such that $\Omega$ satisfies the interior rolling $\delta$-ball condition; equivalently, it can be defined as
$$
\rad(\Omega)=\inf_{x\in\partial\Omega}\sup\{r>0:\text{ there exists } B(z,r)\subset\Omega\text{ tangent to }\partial\Omega\text{ at }x\}.
$$
Finally, below by the injectivity radius $\inj(\Omega)$ of a domain $\Omega\subset M$ we mean the infimum of the injectivity radii $\inj(p)$ of the ambient manifold $M$ as the point $p$ ranges over $\Omega$.

The following theorem is a version of the Gromov and Buser eigenvalue bounds for the Neumann eigenvalue problem. For the convenience of future references we state it in the form of upper bounds for the counting function.
\begin{theorem}
\label{neumann:gromov}
Let $(M,g)$ be a complete Riemannian manifold whose Ricci curvature is bounded below $\Ric\geqslant -(n-1)\kappa$, where $\kappa\geqslant 0$, and $\Omega\subset M$ be a geodesically convex  precompact  domain with Lipschitz boundary. Then the counting function for the Neumann eigenvalue problem on $\Omega$ satisfies inequality~\eqref{leb:gro1}, and hence the eigenvalues satisfy inequality~\eqref{leb:gro}, with $d=d(\Omega)$ being the diameter of the domain. In addition, if the boundary of $\Omega$ is smooth, then  the counting function $N_g(\lambda)$ also satisfies the inequality
\begin{equation}
\label{neumann:volume}
N_g(\lambda)\leqslant C_{5}\vol_g(\Omega)(\lambda^{n/2}+\kappa^{n/2}+\inj(\Omega)^{-n}+\rad(\Omega)^{-n}) \qquad\text{for any}\quad \lambda \geqslant 0,
\end{equation}
where $\rad(\Omega)$ is the maximal radius of an interior rolling ball, $\inj(\Omega)$ is the infimum of the injectivity radii over $\Omega$, and the constant $C_5$ depends on the dimension of $M$ only.
\end{theorem} 
To our knowledge, the equivalent estimates  \eqref{leb:gro} and  \eqref{leb:gro1}  have not been available in the literature for domains with Neumann boundary conditions  under such general assumptions. Previously, Li and Yau~\cite[Theorem 5.3]{LY86} showed that when $M$ has non-negative Ricci curvature and the second fundamental form of $\partial\Omega$ is non-negative definite, the Neumann eigenvalues $\lambda_k(g)$ satisfy the inequalities $\lambda_k(g)\geqslant C\cdot d^{-2}k^{2/n}$, with the constant $C$ depending only on the dimension. It is likely that the method in~\cite{LY86} can be also used to get eigenvalue lower bounds under more general hypotheses, but probably with a more implicit dependence on the diameter and the lower Ricci curvature bound, cf.~\cite{DL82,Wang97}. 

The convexity hypothesis on a domain $\Omega\subset M$ in estimates~\eqref{leb:gro} and~\eqref{leb:gro1}  can not be easily removed. Indeed, consider a Euclidean domain obtained from a  disjoint union of small balls connected by even tinier  passages. As the size of the balls tends to zero and their number increases, so that the domain remains to be contained in a ball of fixed radius, the number of eigenvalues close to zero tends to infinity, while the diameter remains bounded.

To our knowledge,  inequality~\eqref{neumann:volume} for the counting function is new even if  $M$ is a Euclidean space. Examples, obtained by smoothing long thin rectangles, show that  it fails to hold if the interior rolling ball radius $\rad(\Omega)$ is removed on the right-hand side.  However, if the manifold $M$ has a finite volume, then $\rad(\Omega)$ can be dispensed at the price of replacing the volume $\vol(\Omega)$ by the total volume $\vol(M)$, see Remark~\ref{finite}.

Finally, by the variational principle the Neumann eigenvalues are always not greater than the corresponding Dirichlet eigenvalues, and  therefore inequalities~\eqref{leb:gro1} and~\eqref{neumann:volume} hold also for the Dirichlet counting function under the assumption that $\Omega$ is geodesically convex.  A different upper bound for the Dirichlet counting function can be found in~\cite{CY}; it is a generalization of the classical results by Berezin~\cite{Be} and Li and Yau~\cite{LY83} to the setting of eigenvalue problems on Riemannian manifolds.  The bound does not assume that $\Omega$ is convex,  but involves a less explicit geometric quantity,  which could be expressed in terms of  the mean curvatures of $\Omega$ with respect to  isometric embeddings of $M$ into a Euclidean space. It is worth mentioning that  upper bounds on the eigenvalue counting function are important in  applications, such as image processing  and machine learning~\cite{JMS}.

\subsection{Upper eigenvalue bounds: extensions of Cheng and Buser}
Now we discuss the upper eigenvalue bounds on Riemannian manifolds with a lower Ricci curvature bound. We start with recalling classical results due to Cheng~\cite{Che75} and Buser~\cite{Bu79} for the closed eigenvalue problem.
\begin{theorem}
\label{ueb:closed}
Let $(M,g)$ be a closed Riemannian manifold whose Ricci curvature satisfies the bound $\Ric\geqslant -(n-1)\kappa$, where $\kappa\geqslant 0$. Then there exist constants $C_6$ and $C_7$ depending on the dimension $n$ of $M$  only,  such that:
\begin{equation}
\label{cheng}
\lambda_k(g)\leqslant \frac{(n-1)^2}{4}\kappa+C_6(k/d)^2\qquad\text{for any}\quad k\geqslant 1,
\end{equation}
and
\begin{equation}
\label{buser79}
\lambda_k(g)\leqslant \frac{(n-1)^2}{4}\kappa+C_{7}(k/\vol_g(M))^{2/n}\qquad\text{for any}\quad k\geqslant 1,
\end{equation}
where $d=d(M)$ is the diameter of $M$.
\end{theorem}
When a manifold $(M,g)$ has non-negative Ricci curvature, Cheng proves the version of inequality~\eqref{cheng} with an explicit constant:
$$
\lambda_k(g) \leqslant \frac{4k^2 j^2_{\frac{n}{2}-1}}{d^2} < \frac{2k^2 n(n+4)}{d^2},
$$
where $j_{\frac{n}{2}-1}$ is the first zero of the Bessel function $J_{\frac{n}{2}-1}$. The striking difference about the eigenvalue inequalities in Theorem~\ref{ueb:closed} is that the power of $k$ in the second is optimal in the sense of  Weyl's law, while in the first it is not. Nevertheless, as can be seen from the examples of thin flat tori, the quadratic growth in Cheng's inequality~\eqref{cheng} can not be improved.  Note that inequalities similar to~\eqref{buser79} have been also obtained by Li and Yau~\cite{LY80} under somewhat stronger hypotheses. 

Buser's inequality~\eqref{buser79} has been generalised by Colbois and Maerten~\cite{CM08} to the Neumann eigenvalues for compact domains in complete manifolds with a lower Ricci curvature bound. More precisely, they show that there exist constants $C_{8}$ and $C_{9}$ depending on the dimension only such that for any compact domain $\Omega\subset M$ with a Lipschitz boundary the Neumann eigenvalues of $\Omega$ satisfy 
\begin{equation}
\label{cm}
\lambda_k(g)\leqslant C_{8}\kappa+C_{9}(k/\vol_g(\Omega))^{2/n}\qquad\text{for any}\quad k\geqslant 1.
\end{equation}
To complete the picture of eigenvalue upper bounds for the Neumann problem, in Section~\ref{proofs:upper} we prove the following version of Cheng's inequality~\eqref{cheng}.
\begin{theorem}
\label{ueb:np}
Let $(M,g)$ be a complete Riemannian manifold whose Ricci curvature is bounded below $\Ric\geqslant -(n-1)\kappa$, where $\kappa\geqslant 0$, and $\Omega\subset M$ be a geodesically convex  precompact  domain with Lipschitz boundary. Then there exists a constant $C_{10}$ depending on dimension $n$ of $M$ only such that the following inequality for the Neumann eigenvalues of $\Omega$ holds:
\begin{equation}
\label{ourcheng}
\lambda_k(g)\leqslant C_{10}(\kappa+(k/d)^2)\qquad\text{for any}\quad k\geqslant 1,
\end{equation}
where $d=d(\Omega)$  is the diameter of $\Omega$.
\end{theorem}
In the case when $\Omega$ is a convex Euclidean domain, inequality~\eqref{ourcheng} has been obtained in~\cite{Kr}. As the following example shows the convexity hypothesis on a domain $\Omega$ in the theorem above can not be easily removed. First, note that when a domain $\Omega$ is non-convex, its diameter can be also measured using the so-called intrinsic distance on $\Omega$. Recall that it is defined as the infimum of the lengths of paths that lie in $\Omega$ and join two given points. 
\begin{example}
\label{counter}
For a given real number $R>0$  consider a surface of revolution
$$
\Sigma_R=\{(x,y,z)\in\mathbb R^3: y^2+z^2=e^{-2xR}/R^2, x\in [0,1]\}.
$$
As is shown in~\cite[Lemma~5.1]{FT00}, the first non-zero Neumann eigenvalue of $\Sigma_R$ satisfies the inequality $\lambda_1(\Sigma_R)\geqslant R^2/8$. Hence, for the first eigenvalue of the product $\Sigma_R\times [0,\delta]$ we have
$$
%\label{aux:lambda}
\lambda_1(\Sigma_R\times [0,\delta])\geqslant R^2/8\qquad\text{when }0<\delta\leqslant \sqrt{8}\pi R^{-1}.
$$
Now for a sufficiently small $\delta>0$ consider a Euclidean domain in $\mathbb R^3$
$$
\Omega_R(\delta)=\{\exp_p(tv): p\in\Sigma_R, v\text{ is a unit outward normal vector}, ~t\in[0,\delta]\},
$$
where $\exp$ denotes the exponentional map in $\mathbb R^3$. Clearly, it is quasi-isometric to the Riemannian product $\Sigma_R\times [0,\delta]$, and the quasi-isometry constant converges to $1$ as $\delta\to 0+$. Thus, for any sequence $R_\ell\to +\infty$ we may choose a sequence $\delta_\ell\to 0+$ such that the first eigenvalues of  the domains $\Omega_\ell=\Omega_{R_\ell}(\delta_\ell)$ satisfy the inequality $\lambda_1(\Omega_\ell)\geqslant R_\ell^2/16$. Note that the extrinsic diameter of $\Sigma_R$, and hence of any domain containing it, is always greater than $1$. In particular, the extrinsic diameters of $\Omega_\ell$ are bounded away from zero, and we obtain a counterexample to inequality~\eqref{ourcheng} for non-convex Euclidean domains in $\mathbb R^3$, independently of whether the notion of extrinsic or intrinsic diameter is used. It is straightforward to construct other examples of Euclidean domains in $\mathbb R^n$, where $n\geqslant 3$, with similar properties. As was mentioned to us by A. Savo~\cite{S}, there are also examples of non-convex planar domains for which Cheng's upper bound~\eqref{ourcheng} fails. All these examples are closely related to the concentration of measure phenomenon for large eigenvalues, see~\cite{CS11} for details.
\end{example}
Now we state the version of Theorem~\ref{ueb:closed} for the Dirichlet eigenvalue problem, which to our knowledge, appears to be missing in the literature. It involves the maximal radius $\rad(\Omega)$ of an interior rolling ball, and holds for domains with smooth boundary that are not necessarily convex. 
\begin{theorem}
\label{ueb:dp}
Let $(M,g)$ be a complete Riemannian manifold whose Ricci curvature is bounded below $\Ric\geqslant -(n-1)\kappa$, where $\kappa\geqslant 0$, and $\Omega\subset M$ be a precompact domain with smooth boundary. Then there exist constants $C_i$, where $i=11,\dots, 14$ depending on the dimension only such that the Dirichlet eigenvalues  $\nu_k(\Omega)$ satisfy the following inequalities:
\begin{equation}
\label{cheng:dp}
\nu_k(\Omega)\leqslant C_{11}(\kappa+\rad^{-2})+C_{12}((k+1)/\bar d)^2\qquad\text{for any}\quad k\geqslant 0,
\end{equation}
and 
\begin{equation}
\label{buser:dp}
\nu_k(\Omega)\leqslant C_{13}(\kappa+\rad^{-2})+C_{14}((k+1)/\vol(\Omega))^{2/n}\qquad\text{for any}\quad k\geqslant 0,
\end{equation}
where $\rad=\rad(\Omega)$ is the maximal radius of an interior rolling ball,  and $\bar d=\bar d(\Omega)$ is the intrinsic diameter of $\Omega$.
\end{theorem}
Since the extrinsic diameter $d(\Omega)$ is not greater than the intrinsic diameter $\bar d(\Omega)$, we conclude that estimate~\eqref{cheng:dp} holds also for the former in the place of the latter. The examples obtained by rounding long thin rectangles in the Euclidean plane  show that the inequalities in the theorem above  fail to hold even for convex domains if the quantity $\rad(\Omega)$ on the right-hand side is removed. If a domain $\Omega$ has corners, and thus $\rad(\Omega)=0$, Theorem \ref{ueb:dp} can be applied to any smooth domain contained inside $\Omega$, yielding upper bounds on $\nu_k(\Omega)$ using the domain monotonicity.

It is important to mention that the upper bounds for the Dirichlet eigenvalues in Theorem~\ref{ueb:dp} are also upper bounds for the Neumann eigenvalues. In particular, inequality~\eqref{cheng:dp} for the Neumann eigenvalues can be viewed as a version of~\eqref{ourcheng} for non-convex domains; due to Example~\ref{counter} the quantity $\rad(\Omega)$ is necessary. On the other hand, inequality~\eqref{buser:dp} does not give anything new for the Neumann problem, since a stronger inequality~\eqref{cm} due to Colbois and Maerten holds. 

The proofs of Theorems~\ref{ueb:np} and~\ref{ueb:dp} follow the original strategy, used by Cheng and Buser, and are based on versions of volume  comparison theorems. They appear in Section~\ref{proofs:upper}.

\subsection{Multiplicity bounds and related open problems}  
Recall that a classical result due to Cheng~\cite{Che76} says that the multiplicities $m_k(g)$ of the Laplace eigenvalues $\lambda_k(g)$ on a closed Riemannian surface are bounded in terms of the index $k$ and the topology of the surface. The estimate obtained by Cheng has been further improved by Besson~\cite{Be80} and Nadirashvili~\cite{Na87}, and since then 
related questions have been studied extensively in the literature, see~\cite{CC,CdV86,HMN,KKP,GK} and references therein for further details. Note that even the fact that eigenvalue multiplicities on Riemannian surfaces of fixed topology are bounded is by no means trivial, and due to the results of Colin de Verdi\`ere~\cite{CdV}, fails in higher dimensions. More precisely, in dimension $n\geqslant 3$ for any closed manifold $M$ any finite part of the spectrum can be prescribed by choosing an appropriate Riemmannian metric. 

The purpose of the remaining part of the section is to discuss multiplicity bounds for Laplace eigenvalues in terms of geometric quantities, which seem to have been unnoticed in the literature. Recall that by the definition of the counting function, the multiplicity $m_k(g)$ of the Laplace eigenvalue $\lambda_k(g)$ satisfies the inequality $m_k(g)\leqslant N_g(\lambda_k+0)$. Thus, the combination of upper bounds for the counting function and the upper bounds for the Laplace eigenvalues yields the desired bounds for the multiplicities. For the convenience of references we state them below in the form of corollaries, considering the cases of the closed, Neumann, and Dirichlet eigenvalue problems consecutively. The first statement follows by combination of Theorem~\ref{gromov}, or rather inequalities~\eqref{leb:gro1} and~\eqref{improved}, with Theorem~\ref{ueb:closed}.
\begin{cor}
\label{mbc}
Let $(M,g)$ be a closed Riemannian manifold whose Ricci curvature satisfies the bound $\Ric\geqslant -(n-1)\kappa$, where $\kappa\geqslant 0$. Then there exist constants $C_{15}$ and $C_{16}$ depending on the dimension $n$ of $M$  only,  such that the multiplicities $m_k(g)$ of the Laplace eigenvalues $\lambda_k(g)$ satisfy the inequalities
\begin{equation}
\label{mbc:diameter}
m_k(g)\leqslant C_{15}^{1+d\sqrt{\kappa}}(d\sqrt{\kappa}+k^n)\qquad\text{for any }k\geqslant 1,
\end{equation}
and
\begin{equation}
\label{mbc:volume}
m_k(g)\leqslant C_{16}(k+\vol_g(M)(\kappa^{n/2}+\inj^{-n}))\qquad\text{for any }k\geqslant 1,
\end{equation}
where $d$ and $\inj$ are the diameter and the injectivity radius of $M$ respectively.
\end{cor}
As a direct consequence of inequality~\eqref{mbc:diameter}, we see that for manifolds of non-negative Ricci curvature the multiplicities $m_k(g)$ are bounded in terms of the index $k$ and the dimension only. In this statement the hypothesis $\kappa=0$  can not be replaced by a weaker assumption $\kappa>0$, that is by a negative lower Ricci curvature bound. Indeed, this follows from the prescription results~\cite{CdV} together with the fact that the multiplicities $m_k(g)$ are invariant under scaling of a metric $g$. In a similar vein, Lohkamp~\cite{Lohk} shows that any finite part of spectrum can be prescribed by choosing an appropriate Riemannian metric whose volume can be normalised $\vol_g(M)=1$ and the Ricci curvature can be made negative and arbitrarily large in absolute value. This result indicates that the presence of the scale-invariant quantity $\vol(M)\kappa^{n/2}$ in inequality~\eqref{mbc:volume} is rather natural, and one may ask the following question.
\begin{open}
\label{quest1}
Apart from the index $k$ and the dimension, can the multiplicity $m_k(g)$ of a Laplace eigenvalue $\lambda_k(g)$ on a closed manifold $M$ be controlled by the volume and the lower Ricci curvature bound only? 
\end{open}
The inequalities in Corollary~\ref{mbc} have two notable differences. First, the second inequality~\eqref{mbc:volume} is geometry free for a sufficiently large index $k$ in the sense that the second term is dominated by the first one. Second, it is linear in $k$, while the growth in $k$ in inequality~\eqref{mbc:diameter} has order $n$. Concerning the growth of multiplicities in the index $k$, recall that by the result of H\"ormander~\cite{Ho} the sharp remainder estimate in Weyl's law~\eqref{Weyl} is $O(\lambda^{(n-1)/2})$, and hence,  for any given metric $g$ the quantity $m_k(g)k^{(1-n)/n}$ is bounded as $k\to+\infty$. In other words, for a sufficiently large $k$ the multiplicity $m_k(g)$ can not be greater than $C(g)\cdot k^{1-1/n}$, where $C(g)$ is a constant depending on a metric $g$. Though the dependence on the index $k$ in bound~\eqref{mbc:volume} might be satisfactory when the dimension $n$ is large, we ask the following question.
\begin{open}
In inequality~\eqref{mbc:volume} is the linear growth in $k$ the best possible? Can it  be replaced by $ k^{1-1/n}$, where $n$ is the dimension of $M$?
\end{open}
Now we state a version of Corollary~\ref{mbc} for the Neumann eigenvalue problem. It is a consequence of Theorems~\ref{neumann:gromov} and~\ref{ueb:np}, and inequality~\eqref{cm}.
\begin{cor}
\label{mbn}
Let $(M,g)$ be a complete Riemannian manifold whose Ricci curvature is bounded below $\Ric\geqslant -(n-1)\kappa$, where $\kappa\geqslant 0$, and $\Omega\subset M$ be a geodesically convex  precompact  domain with Lipschitz boundary. Then the multiplicities $m_k(g)$ of the Neumann eigenvalue problem on $\Omega$ satisfy inequality~\eqref{mbc:diameter}, with $d=d(\Omega)$ being the  diameter of the domain. In addition, if the boundary of $\Omega$ is smooth, then the multiplicities $m_k(g)$ also satisfy the inequality
\begin{equation}
\label{mbn:volume}
m_k(g)\leqslant C_{17}(k+\vol_g(\Omega)(\kappa^{n/2}+\inj^{-n}+\rad^{-n}))\qquad\text{for any }k\geqslant 1,
\end{equation}
where $\rad(\Omega)$ is the maximal radius of an interior rolling ball, $\inj(\Omega)$ is the infimum of the injectivity radii over $\Omega$, and the constant $C_{17}$ depends on the dimension of $M$ only.
\end{cor}
Following the discussion above for the eigenvalue problem on a closed manifold, we see that the multiplicities $m_k(g)$ of the Neumann eigenvalues of any geodesically convex domain $\Omega$ in the manifold of non-negative Ricci curvature are bounded in terms of the index $k$ and the dimension $n$ only. This statement, and hence also inequality~\eqref{mbc:diameter}, is false without the convexity assumption: indeed, by~\cite{CdV} in dimension $n\geqslant 3$ one can construct Euclidean domains with arbitrary high multiplicities of Neumann eigenvalues.  

We end this section with a discussion of the multiplicity bounds for the Dirichlet eigenvalue problem. The following statement is a consequence of Theorems~\ref{neumann:gromov} and~\ref{ueb:dp}.
\begin{cor}
\label{mbd}
Let $(M,g)$ be a complete Riemannian manifold whose Ricci curvature is bounded below $\Ric\geqslant -(n-1)\kappa$, where $\kappa\geqslant 0$, and $\Omega\subset M$ be a geodesically convex precompact domain with smooth boundary. Then there exist constants $C_{18}$ and $C_{19}$ depending on the dimension only such that the multiplicities $m_k(g)$ of the Dirichlet eigenvalues  $\nu_k(g)$ satisfy the following inequalities:
\begin{equation}
\label{mbd:diameter}
m_k(g)\leqslant C_{18}^{1+d\sqrt{\kappa}}((d\sqrt{\kappa})^n+(d/\rad)^n+k^n)
\end{equation}
and 
\begin{equation}
\label{mbd:volume}
m_k(g)\leqslant C_{19}(k+1+\vol(\Omega)(\kappa^{n/2}+\inj^{-n}+\rad^{-n})),
\end{equation}
where $\rad=\rad(\Omega)$ is the maximal radius of an interior rolling ball,  and $d=d(\Omega)$ is the diameter of $\Omega$.
\end{cor}
Note that all multiplicity bounds in the corollaries above are in fact bounds for the sums $\sum_{i\leqslant k} m_i(g)$, and in particular, may not reflect the actual behaviour of the individual multiplicities. It is plausible that in particular instances the multiplicities satisfy better bounds. For example, considering inequality~\eqref{mbd:diameter} for Euclidean domains, one can ask whether the remaining dependence on geometry is actually necessary. 
\begin{open}
\label{open:1}
Does there exist a constant $C(n,k)$ depending on the dimension $n\geqslant 3$ and the index $k \geqslant 1$, such that the multiplicity of the $k$-th Dirichlet eigenvalue of a Euclidean domain $\Omega\subset\mathbb R^n$ is bounded above by $C(n,k)$?
\end{open}
Clearly, $C(n,0)=1$ for all $n$,  and by the results of~\cite{Na87}, see also~\cite{HMN, KKP, Berd},  one can take $C(2,k)=2k+1$ for $k\geqslant 1$. To our knowledge, the question above is open even for convex domains, where we have a positive answer for the Neumann problem, see the first statement in Corollary~\ref{mbn}. If  instead of Euclidean domains we consider arbitrary Riemannian manifolds with boundary, then the answer to  Open Problem \ref{open:1} is negative. 
Indeed, by~\cite{CC, CdV} for any integers $n\geqslant 2$, $k\geqslant 1$,  and  $N>0$  there exists a closed manifold $M$ of dimension $n$, such that  $m_{k}(M)>N$.  Then, for a sufficiently small $\epsilon>0$  the multiplicity of the $k$-th Dirichlet eigenvalue of the cylinder $M \times [-\epsilon, \epsilon]$, equipped with the product metric, also satisfies $m_{k}(M)>N$.

Note also that the methods used in~\cite{CC, CdV} to construct closed surfaces with Laplace eigenvalues of high multiplicity can be generalized directly to surfaces with Neumann boundary conditions.  However, the approach does not extend in a straightforward way to the case of the Dirichlet boundary conditions. It would be interesting to know whether for any $k\geqslant 1$ there exists a surface with boundary whose $k$th Dirichlet eigenvalue has  an arbitrary large multiplicity.

In higher dimensions the Dirichlet eigenvalues also behave differently: they satisfy the so-called universal inequalities, and hence, there is no analogue of the eigenvalue prescription results~\cite{CdV} for this problem. Nevertheless, it is still possible that the multiplicities can be prescribed; we state this question in the form of the following problem.
\begin{open}
Let $M$ be a manifold with boundary of dimension $n \geqslant 3$. For given integers $k\geqslant 1$ and $N\geqslant 1$ does there exist a Riemannian metric on $M$ such that the multiplicity of the $k$-th Dirichlet eigenvalue is equal to $N$?
\end{open}

We conclude with a few remarks on multiplicity bounds similar to inequality~\eqref{mbd:volume}. Recall that for a convex Euclidean domain $\Omega\subset\mathbb R^n$ it takes the form
$$
m_k(g)\leqslant C_{19}(k+1+\vol(\Omega)/\rad^n).
$$
For arbitrary precompact Euclidean domains one can also bound the multiplicity in terms of volume and inradius; the latter quantity is defined as the maximal radius of an inscribed ball
$$
\rho(\Omega)=\sup\{r: B(x,r)\subset \Omega~\text{for some $x\in\Omega$}\}.
$$
In more detail, by the result of Li and Yau~\cite{LY83} the Dirichlet counting function of an arbitrary domain $\Omega\subset\mathbb R^n$ satisfies the inequality $N(\lambda)\leqslant C_{20}\vol(\Omega)\lambda^{n/2}$, where $C_{20}$ is a constant depending only on the dimension. Combining this inequality with the upper bound due to Cheng and Yang~\cite[Proposition 3.1]{CY07}:
$$
\nu_k(\Omega)\leqslant \frac{n+3}{n}\nu_0(\Omega)(k+1)^{2/n}\qquad\text{for any }k\geqslant n,
$$
we obtain
\begin{multline*}
m_k(g)\leqslant C_{21}\vol(\Omega)\nu_0(\Omega)^{n/2}(k+1)\leqslant C_{21}\vol(\Omega)\nu_0(B({\rho}))^{n/2}(k+1)\\ \leqslant C_{22}\left(\vol(\Omega)/\rho(\Omega)^n\right)(k+1),
\end{multline*}
where $B({\rho})$ is an inscribed ball of radius $\rho=\rho(\Omega)$, and in the second inequality we used the domain monotonicity. Note that a similar multiplicity bound for Neumann eigenvalues does not hold if $n\geqslant 3 $, as one can  prescribe any finite part of the Neumann spectrum while keeping the volume and the inradius of a domain bounded. The last statement can be deduced  by inspecting the arguments in~\cite[pp. 610-611]{CdV}.

%\subsection{Epilogue: curvature free eigenvalue inequalities}

%
%---------------------------------------------------
%                     Section 2
%---------------------------------------------------
%
\section{Poincar\'e inequality and coverings by metric balls}
\label{prem}
\subsection{{Poincar\'e inequality}}
A key ingredient in our approach to the lower eigenvalue bounds by Gromov and Buser is the following version of the Neumann-Poincar\'e inequality.
\begin{prop}
\label{np:intro}
Let $(M,g)$ be a complete Riemannian manifold whose Ricci curvature is bounded below, $\Ric\geqslant -(n-1)\kappa$, where $\kappa\geqslant 0$ and $n$ is the dimension of $M$. Then, for any $p\geqslant 1$, there exists a constant $C_N=C_N(n,p)$ that depends on the dimension $n$ and $p$ only,  such that for any smooth function $u$ on $M$ the following inequality holds:
$$
\int_{B_R}\abs{u-u_R}^pd\vol\leqslant C_NR^p{e^{(n-1)R\sqrt{\kappa}}}\int_{B_{2R}}\abs{\nabla u}^pd\vol,
$$
where $B_R$ and $B_{2R}$ are concentric metric balls in $M$ of radii $R$ and $2R$ respectively, and $u_R$ is the mean-value of $u$ on $B_R$, i.e. $u_R=\vol(B_R)^{-1}\int_{B_R}u\,d\!\vol$.
\end{prop}
The statement above is folkloric; related results can be found in~\cite[section 5]{Bu82} and~\cite{SC02}. We  extend the above inequality to the case of convex domains in Riemannian manifolds.
\begin{prop}
\label{np4np}
Let $(M,g)$ be a complete Riemannian manifold whose Ricci curvature is bounded below, $\Ric\geqslant -(n-1)\kappa$, where $\kappa\geqslant 0$ and $n$ is the dimension of $M$. Then for any $p\geqslant 1$ there exists a constant $C_N=C_N(n,p)$ that depends on the dimension $n$ and $p$ only such that for any geodesically convex domain $\Omega\subset M$ and for any smooth function $u$ on $\Omega$ the following inequality holds
\begin{equation}\label{poindomain}
\int_{B_R\cap\Omega}\abs{u-u_R}^pd\vol\leqslant C_NR^p{e^{(n-1)R\sqrt{\kappa}}}\int_{B_{2R}\cap\Omega}\abs{\nabla u}^pd\vol,
\end{equation}
where $B_R$ and $B_{2R}$ are concentric metric balls in $M$ of radii $R$ and $2R$ respectively, and $u_R$ is the mean-value of $u$ on $B_R\cap\Omega$, i.e. $u_R=\vol(B_R\cap\Omega)^{-1}\int_{B_R\cap\Omega}ud\vol$.
\end{prop}
The inequality in Proposition \ref{np4np} (with a slightly different constant in the exponent) can be obtained by building on the argument used in~\cite[Chap.~5]{SC02}. Below we give a shorter proof, avoiding technicalities by using the so-called segment inequality due to Cheeger and Colding~\cite{CC96}. Before stating it we introduce the following notation: we set
$$
C(n,\kappa, R)=2R\sup_{0<s/2\leqslant t\leqslant s}\frac{\vol(\partial B_\kappa(s))}{\vol(\partial B_\kappa(t))},
$$
where $R>0$ and $\partial B_\kappa(r)$ is a sphere of radius $r$ in an $n$-dimensional simply connected space of constant sectional curvature $-\kappa$. Note that for $\kappa\geqslant 0$ the ratio of volumes above is not greater than $(s/t)^{n-1}e^{(n-1)s\sqrt{\kappa}}$, and we obtain
\begin{equation}
\label{rem}
C(n,\kappa,R)\leqslant 2^nRe^{(n-1)R\sqrt{k}}
\end{equation}
The following proposition is a reformulation of~~\cite[Theorem 2.11]{CC96}.
\begin{prop}[The segment inequality]
\label{segmentinq}
Let $(M,g)$ be a complete Riemannian manifold whose Ricci curvature is bounded below, $\Ric\geqslant -(n-1)\kappa$, where $\kappa\geqslant 0$ and $n$ is the dimension of $M$. Let $B_R$ be a metric ball, $A$ and $B$ be open subsets in $B_R$, and $W\subset M$ be an open subset that contains the convex hull of the union $A\cup B$. Then for any nonnegative integrable function $F$ on $W$ the following inequality holds:
\begin{equation}
\label{cheeger-colding}
\int_{A\times B}\int_0^{d(x,y)}F(\gamma_{x,y}(s))ds\, dx \,dy\leqslant C(n,\kappa,R) (\vol(A)+\vol(B))\int_{W} F(z) dz,
\end{equation}
where $\gamma_{x,y}:[0,d(x,y)]\to M$ is a shortest  geodesic joining $x$ and $y$, and the first integral on the left hand-side is taken over the subset of $A\times B$ formed by the pairs $(x,y)$ of points that can be joined by such a unique geodesic.
\end{prop}
\begin{proof}[Proof of Proposition \ref{np4np}]
For arbitrary open subsets $A$ and $B$ consider the set of pairs $(x,y)\in A\times B$ such that the points $x$ and $y$ can be joined by a unique shortest geodesic $\gamma_{x,y}$. By standard results in Riemannian geometry, see~\cite{Cha}, its complement in $A\times B$ has zero measure, and abusing the notation, we also denote it below by $A\times B$.

It is not hard to see that for any $x\in M$ the inequality
$$
\abs{u-u_R}^p(x)\leqslant\vol(B_R\cap\Omega)^{-1}\int_{B_R\cap\Omega}\abs{u(x)-u(y)}^pdy.
$$
holds, where  $u_R=\vol(B_R\cap\Omega)^{-1}\int_{B_R\cap\Omega}u$. Indeed, for $p=1$ it is straightforward, and for $p>1$ it can be obtained from the former case by using the H\"older inequality. Integrating it over $B_R\cap\Omega$, we obtain
\begin{multline*}
\int_{B_R\cap\Omega}\abs{u(x)-u_R}^pdx\leqslant\vol(B_R\cap\Omega)^{-1}\int_{B_R\cap\Omega}\int_{B_R\cap\Omega}\abs{u(x)-u(y)}^pdxdy\\
\leqslant\vol(B_R\cap\Omega)^{-1}\int_{(B_R\cap\Omega)\times (B_R\cap\Omega)}\left(\int_0^{d(x,y)}\abs{\nabla u(\gamma_{x,y}(s))}ds\right)^pdxdy\\
\leqslant (2R)^{p-1}\vol(B_R\cap\Omega)^{-1}\int_{(B_R\cap\Omega)\times(B_R\cap\Omega)}\int_0^{d(x,y)}\abs{\nabla u(\gamma_{x,y}(s))}^pdsdxdy,
\end{multline*}
where in the last inequality we used the H\"older inequality and the relation $\dist(x,y)\leqslant 2R$. Since $\Omega$ is convex, the convex hull of $B_R\cap\Omega$ lies in $B_{2R}\cap\Omega$. Thus, applying Proposition~\ref{segmentinq} with $A=B=B_R\cap\Omega$ and $W=B_{2R}\cap\Omega$, and using inequality~\eqref{rem}, we obtain
$$
\int_{(B_R\cap\Omega)\times (B_R\cap\Omega)}\int_{0}^{d(x,y)}\abs{\nabla u(\gamma_{x,y}(s))}^p ds\,dx\, dy\leqslant 2^{n+1}e^{(n-1)R\sqrt{\kappa}}R\vol(B_R\cap\Omega)\int_{B_{2R}\cap\Omega}\abs{\nabla u(z)}^pdz.
$$
Combining the last two inequalities, we arrive at the Poincare inequality \eqref{poindomain}.
\end{proof}

\subsection{{Coverings by metric balls}: closed manifolds}
We proceed with the estimates for the cardinality and multiplicity of certain coverings. The following lemma is by now a standard application of the Gromov-Bishop volume comparison theorem. We outline its proof for the sake of completeness.
\begin{lemma}
\label{cp:l1}
Let $(M,g)$ be a closed Riemannian manifold whose Ricci curvature satisfies the bound $\Ric\geqslant -(n-1)\kappa$, where $\kappa\geqslant 0$. Let $(B_i)$ be a covering of $M$ by balls $B_i=B(x_i,\rho)$ such that the balls $B(x_i,\rho/2)$ are disjoint. Then:
\begin{itemize}
\item[(i)] for any $0<\rho\leqslant 2d$ the cardinality of the family $(B_i)$ is not greater than $2^ne^{(n-1)d\sqrt{\kappa}}(d/\rho)^n$, where $d$ is the diameter of $M$;
\item[(ii)] for any $\rho>0$ and for any $x\in M$ the number of balls from $(B(x_i,2\rho))$ that contain $x$ is not greater than ${12}^ne^{6(n-1)\rho\sqrt{\kappa}}$.
\end{itemize}
\end{lemma}
\begin{proof}
First, by the relative volume comparison theorem, see~\cite{Cha}, it is straightforward to show that the volumes of concentric metric balls of radii $0<r\leqslant R$ satisfy the relation
\begin{equation}
\label{vc}
\vol(B_R)\leqslant e^{(n-1)R\sqrt{\kappa}}(R/r)^n\vol(B_r).
\end{equation}
Now to prove~$(i)$ note that $m=\card(B_i)$ satisfies the following relations
$$
m\cdot\inf_i\vol(B(x_i,\rho/2))\leqslant\sum_i\vol(B(x_i,\rho/2))\leqslant\vol(M).
$$
Let $x_{i_0}$ be a point at which the infimum in the left hand-side above is achieved. Then, for any $0<\rho\leqslant 2d$ we obtain
$$
m\leqslant\vol(B(x_{i_0},d))/\vol(B(x_{i_0},\rho/2))\leqslant 2^ne^{(n-1)d\sqrt{\kappa}}(d/\rho)^n,
$$
where in the last inequality we used~\eqref{vc}. 

To prove the statement~$(ii)$ we re-denote by $x_{i_0}$ the point at which the infimum $\inf\vol(B(x_i,\rho/2))$ is achieved while $i$ ranges over all indices such that the balls $B(x_i,2\rho)$ contain $x$. Note that if $x\in B(x_i,2\rho)$, then $B(x_i,2\rho)\subset B(x_{i_0},6\rho)$. Thus, for any $\rho>0$ we obtain that
$$
\mult_x(B_i)\leqslant\vol(B(x_{i_0},6\rho))/\vol(B(x_{i_0},\rho/2))\leqslant {12}^ne^{6(n-1)\rho\sqrt{\kappa}},
$$
where in the last inequality we again used~\eqref{vc}.
\end{proof}
For a proof of the Buser inequality in Theorem~\ref{gromov} we  also need the following supplement to Lemma~\ref{cp:l1}.
\begin{lemma}
\label{cp:l2}
Under the hypotheses of Lemma~\ref{cp:l1}, the cardinality of the family $(B(x_i,\rho))$ 
%and the number of balls from $(B(x_i,3\rho))$ that contain a given point $x\in M$ are 
is not greater than $c_1\vol(M)(\min\{\rho,\inj\})^{-n}$, where $\inj$ is the injectivity radius of $M$, and  $c_1$ is a constant that depends on the dimension $n$ only.
\end{lemma}
\begin{proof}
As in the proof of Lemma~\ref{cp:l1}, we see that
$$
m=\card(B_i)\leqslant\vol(M)/\vol(B(x_{i_0},\rho/2))
$$
for some point $x_{i_0}$. Recall that by~\cite[Prop.~14]{Cro80} the volume of a geodesic ball satisfies the inequality
$$
c_2\rho^n\leqslant\vol(B(x,\rho/2))\qquad\text{for any}\quad\rho\leqslant\inj,
$$
where $c_2$ is a constant that depends on $n$ only. For $\rho\geqslant\inj$, we clearly have 
$$
c_2\inj^n\leqslant\vol(B(x,\inj/2))\leqslant\vol(B(x,\rho/2)).
$$
Combining these inequalities with the bound for the cardinality $m$ above, we complete the proof of the lemma. 
\end{proof}

\subsection{Coverings by metric balls: domains}
Now we discuss versions of the above statements for coverings of domains in Riemannian manifolds.
\begin{lemma}
\label{np:l}
Let $(M,g)$ be a complete Riemannian manifold whose Ricci curvature is bounded below, $\Ric\geqslant -(n-1)\kappa$, where $\kappa\geqslant 0$ and $n$ is the dimension of $M$. Let $\Omega\subset M$ be a precompact domain, and $(B_i)$ be its covering by  balls $B_i=B(x_i,\rho)$ such that $x_i\in\Omega$ and the balls $B(x_i,\rho/2)$ are disjoint. Then:
\begin{itemize} 
\item[(i)] if $\Omega$ is convex, the conclusions of Lemma~\ref{cp:l1} hold, where $d=d(\Omega)$ is the (extrinsic) diameter of $\Omega$;
\item[(ii)] if $\Omega$ has a smooth boundary, the cardinality of the covering $(B_i)$ is not greater than $c_3\vol(\Omega)(\min\{\rho,\inj,\rad\})^{-n}$, where $c_3$ is a constant that depends on $n$ only, $\inj=\inj(\Omega)$ is the injectivity radius of $\Omega$, and $\rad=\rad(\Omega)$ is the maximal radius of an inscribed rolling ball; 
\item[(iii)] if $M$ has finite volume, then the cardinality of the covering $(B_i)$ is not greater than $c_4\vol(M)(\min\{\rho,\inj\})^{-n}$.
\end{itemize}
\end{lemma}
In the sequel we use the following folkloric version of Gromov-Bishop relative volume comparison theorem, see~\cite[p.524]{Gro}; we outline its proof for the sake of completeness.
\begin{lemma}
\label{GB}
Let $(M,g)$ be a complete Riemannian manifold whose Ricci curvature is bounded below, $\Ric\geqslant -(n-1)\kappa$, where $\kappa\geqslant 0$, and $\Omega\subset M$ be a precompact domain that is star-shaped with respect to a point $x\in\bar\Omega$. Then the quotient $\vol(B(x,r)\cap\Omega)/\vol(B_\kappa(r))$, where $B_\kappa(r)$ is a ball in the space of constant curvature $(-\kappa)$, is a non-increasing function in $r>0$. In particular, for any $0<r\leqslant R$ we have
$$
\vol(B(x,R)\cap\Omega)\leqslant e^{(n-1)R\sqrt{\kappa}}(R/r)^n\vol(B(x,r)\cap\Omega).
$$
\end{lemma}
\begin{proof}
For a given subset $S$ of a unit sphere $S^{n-1}\subset\mathbb R^n$ denote by $C_S$ the cone $\{\exp_x(t\xi): t>0, \xi\in S\}$. The standard proof of the Gromov-Bishop comparison theorem, see~\cite[p.134-135]{Cha}, shows that the quotient $\vol(B(x,r)\cap C_S)/\vol(B_\kappa(r))$ is a non-increasing function in $r>0$. For a given $0<r\leqslant R$ define $S$ as the set formed by $\xi\in S^{n-1}$ such that $\exp_x(r\xi)\in\Omega$. Since $\Omega$ is star-shaped with respect to $x$, we conclude that:
\begin{itemize}
\item[{\it (a)}] $B(x,r)\cap C_S\subset B(x,r)\cap\Omega$, 
\item[{\it (b)}] $(B(x,R)\backslash B(x,r))\cap\Omega\subset(B(x,R)\backslash B(x,r))\cap C_S$.
\end{itemize}
By relation $(a)$ the quantity
$$
h:=\vol(B(x,r)\cap\Omega)-\vol(B(x,r)\cap C_S) 
$$
is non-negative, and by~$(b)$, we obtain
$$
\vol(B(x,R)\cap\Omega)-h\leqslant\vol(B(x,R)\cap C_S).
$$
Finally, using the Gromov-Bishop theorem for the intersections of balls with cones, we obtain
\begin{multline*}
\vol(B(x,R)\cap\Omega)/\vol(B(x,r)\cap\Omega)\leqslant (\vol(B(x,R)\cap\Omega)-h)/(\vol(B(x,r)\cap\Omega)-h)\\ \leqslant \vol(B(x,R)\cap C_S)/\vol(B(x,r)\cap C_S)\leqslant \vol(B_\kappa(R))/\vol(B_\kappa(r)).
\end{multline*}
The last statement of the lemma follows from the standard formula for the volume $\vol(B_k(r))$, see~\cite{Cha}, which leads to the estimate for the quotient $\vol(B_\kappa(R))/\vol(B_\kappa(r))$.
\end{proof}

\begin{proof}[Proof of Lemma~\ref{np:l}]
Following the argument in the proof of Lemma~\ref{cp:l1}, we see that
\begin{equation}
\label{lemma:aux1}
m=\card(B_i)\leqslant\vol(\Omega)/\vol(B(x_{i_0},\rho/2)\cap\Omega)
\end{equation}
for some point $x_{i_0}\in\Omega$. If $d=d(\Omega)$ is the diameter of $\Omega$, then $\Omega$ lies in the ball $B(x_{i_0},d)$, and by Lemma~\ref{GB}, we obtain
$$
m\leqslant\vol(B(x_{i_0},d)\cap\Omega)/\vol(B(x_{i_0},\rho/2)\cap\Omega)\leqslant 2^ne^{(n-1)d\sqrt{\kappa}}(d/\rho)^n.
$$
The estimate for the multiplicity of the covering $(B_i)$ is the same as in the proof of Lemma~\ref{cp:l1}. We proceed with the statement~(ii): by relation~\eqref{lemma:aux1} for a proof it is sufficient to show that
\begin{equation}
\label{lemma:aux2}
c_5\rho^n\leqslant\vol(B(x,\rho/2)\cap\Omega)\qquad\text{for any}\quad\rho\leqslant\min\{\inj,\rad\}.
\end{equation}
To see that the above relation holds note that for any ball $B(x,r)$, where $x\in\Omega$ and $r< 2\rad(\Omega)$, there exists a point $\tilde x\in B(x,r)$ such that 
$$
\dist(x,\tilde x)< r/2\qquad\text{and}\qquad B(\tilde x,r/2)\subset B(x,r)\cap\Omega.
$$
Indeed, the statement is clear if $B(x,r/2)\subset\Omega$. If $B(x,r/2)$ does not lie entirely in $\Omega$, then since $B(\tilde x,r/2)\subset\Omega$, one can take an inscribed ball that touches the boundary $\partial\Omega$ at a point $p$ where the minimum of the distance $\dist(q,x)$, while $q$ ranges over $\partial\Omega$, is achieved. It is straightforward to see that the point $x$ belongs to the shortest geodesic arc joining $\tilde x$ and $p$, which meets the boundary $\partial\Omega$ at the point $p$ orthogonally.  In particular, the ball $B(\tilde x,r/2)$ is also contained in the ball $B(x,r)$. Thus, we conclude that under our hypotheses
$$
\vol(B(\tilde x,\rho/4))\leqslant\vol(B(x,\rho/2)\cap\Omega),
$$
and by~\cite[Prop.~14]{Cro80} the quantity on the left-hand side is at least $c_5\rho^n$ when $\rho/2<\inj(\Omega)$. Combining the last statement with the hypothesis $\rho/2<\rad(\Omega)$, we prove relation~\eqref{lemma:aux2}.

Under the hypotheses of the last statement of the lemma we may bound the cardinality $m$ of the covering by $\vol(M)/\vol(B(x_{i_0},\rho/2))$, and then appeal directly to Croke's result~\cite[Prop.~14]{Cro80} in the same fashion as in the proof of Lemma~\ref{cp:l2}.
\end{proof}

\section{Lower eigenvalue bounds}
\label{proofs:lower}
\subsection{Proof of Theorem~\ref{gromov}: Gromov's inequalities}
We prove estimate~\eqref{leb:gro1} for the counting function $N_g(\lambda)$.
{For a given real number $\lambda>0$ denote by $E(\lambda)$ the sum of all eigenspaces that correspond to the eigenvalues $\lambda_k(g)<\lambda$. Recall that by the  variational principle, for any $0\ne\varphi\in E(\lambda)$ we have
\begin{equation}
\label{rayleigh}
\int_M\abs{\nabla\varphi}^2d\vol<\lambda\int_M\varphi^2d\vol.
\end{equation}
For a given $\rho>0$ consider a covering of $M$ by balls $B_i=B(x_i,\rho)$  such that the balls $B(x_i,\rho/2)$ are disjoint. It can be obtained by choosing the collection of balls $B(x_i,\rho/2)$ to be a maximal collection of disjoint balls. Given such a covering $(B_i)$ we define the map
$$
\Phi_\lambda:E(\lambda)\to\mathbb R^m,\qquad u\mapsto\vol(B_i)^{-1}\int_{B_i}u\, d\vol,
$$
where $m$ stands for the cardinality of $(B_i)$, and $i=1,\ldots, m$, cf.~\cite{ML}. We claim that there exists a constant $c_6$ depending on the dimension $n$ only such that if $\lambda^{-1}\geqslant c_6\rho^{2}e^{7(n-1)\rho\sqrt{\kappa}}$, then the map $\Phi_\lambda$ is injective. To see this we define $c_6={12}^{n}C_N$, where $C_N=C_N(n,2)$ is the Poincar\'e constant from Proposition~\ref{np:intro}, and argue by assuming the contrary. Suppose that $\varphi\ne 0$ belongs to the kernel of $\Phi_\lambda$. Then we obtain
\begin{multline*}
\int_M\varphi^2d\vol\leqslant\sum_i\int_{B_i}\varphi^2d\vol\leqslant C_N\rho^2e^{(n-1)\rho\sqrt{\kappa}}\sum_i\int_{2B_i}\abs{\nabla\varphi}^2d\vol\\
\leqslant c_6\rho^2e^{7(n-1)\rho\sqrt{\kappa}}\int_M\abs{\nabla\varphi}^2d\vol,
\end{multline*}
where we used Proposition~\ref{np:intro} in the second inequality and Lemma~\ref{cp:l1} in the last. Now combining these relations with~\eqref{rayleigh}, we conclude that $\lambda^{-1}<c_6\rho^{2}e^{7(n-1)\rho\sqrt{\kappa}}$, and arrive at a contradiction. Thus, for a sufficiently small $\rho$ the map $\Phi_\lambda$ is injective, and the value $N_g(\lambda)$ is not greater than the cardinality $m$ of a covering $(B_i)$.}

For a given $\lambda>0$ we set
$$
\rho_0=(c_6\lambda e^{14(n-1)d\sqrt{\kappa}})^{-1/2}, 
$$
where $d$ is the diameter of $M$. When $\rho_0\leqslant 2d$, it is straightforward to check that the relation $\lambda^{-1}\geqslant c_6\rho_0^{2}e^{7(n-1)\rho_0\sqrt{\kappa}}$ holds, and  by Lemma~\ref{cp:l1} we obtain
$$
N_g(\lambda)\leqslant m\leqslant 2^ne^{(n-1)d\sqrt{\kappa}}(d/\rho_0)^n\leqslant C_4^{1+d\sqrt{\kappa}}d^n\lambda^{n/2}.
$$
To treat the case $\rho_0>2d$, note that there is only one covering with balls of radius $\rho\geqslant 2d$ that satisfies our hypotheses, and it consists of only one ball. In particular, if $\rho_0>2d$, then the covering under the consideration coincides with the one for $\rho_*=2d$ for which $\Phi_\lambda$ is also injective. Indeed, by the definition of $\rho_0$, it is straightforward to see that the relation $\rho_0\geqslant 2d$ implies that $\lambda^{-1}\geqslant c_6\rho_*^{2}e^{7(n-1)\rho_*\sqrt{\kappa}}$. Since such a  covering consists of only one ball, we conclude that in this case $N_g(\lambda)$ is not greater than $1$. Combining these two cases, we finish the proof of the theorem.\qed
\begin{remark}
The idea to use the bounds for the first eigenvalue on small sets to get estimates for higher eigenvalues is not new; see, for example, the already mentioned papers by Cheng~\cite{Che75}, Gromov~\cite[Appendix~C]{Gro}, and Li and Yau~\cite{LY80}. A similar strategy has been used in~\cite{ML} in the context of the multiplicity bounds for Laplace eigenvalues. Note that the eigenvalue multiplicity bound for closed manifolds of non-negative Ricci curvature obtained in~\cite{ML}, see formula~$(6)$ in~\cite[Theorem~3.1]{ML}, is a partial case of~\eqref{mbc:diameter},  which is a consequence of the results of Cheng and Gromov cited above.
\end{remark}

\subsection{Proof of Theorem~\ref{gromov}: Buser's inequalities}
Consider a covering of $M$ by balls $B_i=B(x_i,\rho)$  such that $B(x_i,\rho/2)$ form a maximal family of disjoint balls. As is shown above, if a real number $\lambda>0$ satisfies the inequality $\lambda^{-1}\geqslant c_6\rho^{2}e^{7(n-1)\rho\sqrt{\kappa}}$, where $c_6$ depends on $n$ only, then $N_g(\lambda)$ is not greater than the cardinality $m=\card(B_i)$. In this case by Lemma~\ref{cp:l2} we have
\begin{equation}
\label{eq:count}
N_g(\lambda)\leqslant m\leqslant c_1\vol(M)(\rho^{-n}+\inj^{-n}).
\end{equation}
The hypothesis on $\rho$ is clearly satisfied, if $\lambda^{-1}\geqslant 2c_6\rho^2$ {and} $2\geqslant e^{7(n-1)\rho\sqrt{\kappa}}$. Thus, choosing $\rho=\rho_0$ as the minimum of the values $(2c_6\lambda)^{-1/2}$  and $\ln 2(7(n-1)\sqrt{\kappa})^{-1}$, by relation~\eqref{eq:count} we obtain
$$
N_g(\lambda)\leqslant C_3\vol(M)(\lambda^{n/2}+\kappa^{n/2}+\inj^{-n}),
$$
where the value of the constant $C_3$ depends on $c_1$, $c_6$, and the dimension $n$.
\qed

\subsection{Proof of Theorem~\ref{neumann:gromov}: Gromov's inequalities for domains}
The proof of estimate \eqref{leb:gro1} for convex domains follows a line of argument similar to the one in the proof of Theorem~\ref{gromov}; it uses the Neumann-Poincar\'e inequality in Proposition~\ref{np4np} and Lemma~\ref{np:l}.

More precisely, for $\lambda>0$ denote by $E(\lambda)$ the sum of all eigenspaces that correspond to the Neumann eigenvalues $\lambda_k(g)<\lambda$. Let $(B_i)$ be a covering of $\Omega$ by balls $B_i=B(x_i,\rho)$ such that $x_i\in\Omega$ and the smaller balls $B(x_i,\rho/2)$ are disjoint. We claim that if $\lambda^{-1}\geqslant c_6\rho^{2}e^{7(n-1)\rho\sqrt{\kappa}}$, then the map
$$
\Phi_\lambda:E(\lambda)\to\mathbb R^m,\qquad u\mapsto\vol(B_i\cap\Omega)^{-1}\int_{B_i\cap\Omega}ud\vol,
$$
is injective, and the value $N_g(\lambda)$ is not greater than $m=\card(B_i)$. Indeed, suppose that a function $\varphi\ne 0$ belongs to the kernel of $\Phi_\lambda$. Then, setting $c_6={12}^nC_N$ with $C_N=C_N(n,2)$ being the constant from Proposition~\ref{np4np}, we obtain
\begin{multline*}
\int_\Omega\varphi^2d\vol\leqslant\sum_i\int_{B_i\cap\Omega}\varphi^2d\vol\leqslant C_N\rho^2e^{(n-1)\rho\sqrt{\kappa}}\sum_i\int_{2B_i\cap\Omega}\abs{\nabla\varphi}^2d\vol\\
\leqslant c_6\rho^2e^{7(n-1)\rho\sqrt{\kappa}}\int_\Omega\abs{\nabla\varphi}^2d\vol,
\end{multline*}
where in the last relation we used Lemma~\ref{np:l}. Now we arrive at a contradiction in the same fashion as above.

For a given $\lambda>0$ we set
$$
\rho_0=(c_6\lambda e^{14(n-1)d\sqrt{\kappa}})^{-1/2}, 
$$
where $d$ is the diameter of $\Omega$. When $\rho_0\leqslant 2d$, it is straightforward to check that the above hypothesis on $\lambda$ holds, and the value $N_g(\lambda)$ is bounded by the cardinality of the 
covering $B(x_i,\rho_0)$ such that $x_i\in\Omega$ and the balls $B(x_i,\rho_0/2)$ are disjoint. Then, by Lemma~\ref{np:l} we obtain
$$
N_g(\lambda)\leqslant m\leqslant 2^ne^{(n-1)d\sqrt{\kappa}}(d/\rho_0)^n\leqslant C_{4}^{1+d\sqrt{\kappa}}d^n\lambda^{n/2}.
$$
The case $\rho_0>2d$ is treated in the fashion similar to the one in the proof of Theorem~\ref{gromov}.
\qed

\subsection{Proof of Theorem~\ref{neumann:gromov}: Buser's inequalities for domains}
Let $(B_i)$ be a covering of $\Omega$ by balls $B_i=B(x_i,\rho)$ such that $x_i\in\Omega$ and the smaller balls $B(x_i,\rho/2)$ form a maximal family of disjoint balls. As is shown in the proof of Gromov's inequalities for Neumann eigenvalues, if $\lambda^{-1}\geqslant c_6\rho^{2}e^{7(n-1)\rho\sqrt{\kappa}}$, then the value $N_g(\lambda)$ is not greater than $m=\card(B_i)$. Now by Lemma~\ref{np:l} we have
$$
N_g(\lambda)\leqslant m\leqslant c_3\vol(\Omega)(\rho^{-n}+\inj^{-n}+\rad^{-n}).
$$ 
Choosing $\rho=\rho_0$ as the minimum of the values $(2c_6\lambda)^{-1/2}$  and $\ln 2(7(n-1)\sqrt{\kappa})^{-1}$, we obtain the desired bound on the counting function.
\qed
\begin{remark}
\label{finite}
When a manifold $M$ has a finite volume, the argument above yields the estimate
$$
N_\Omega(\lambda)\leqslant C\cdot\vol_g(M)(\lambda^{n/2}+\kappa^{n/2}+\inj(\Omega)^{-n}) \qquad\text{for any}\quad \lambda \geqslant 0,
$$
for the Neumann eigenvalues counting function of any compact geodesically convex domain $\Omega\subset M$. Indeed, this is a consequence of the following estimate for the cardinality $m$ of the covering $(B_i)$ with the properties described above:
$$
m=\card(B_i)\leqslant c_4\vol(M)(\rho^{-n}+\inj^{-n}),
$$ 
see Lemma~\ref{np:l}.
\end{remark}

\section{Upper eigenvalue bounds}
\label{proofs:upper}
\subsection{Proof of Theorem~\ref{ueb:np}}
Since $\Omega\subset M$ is geodesically convex, by Lemma~\ref{GB} for any $x\in\bar\Omega$ and any $0<r\leqslant 1/\sqrt{\kappa}$ we have
\begin{equation}
\label{gb}
\vol(B(x,r)\cap\Omega)/\vol(B(x,r/2)\cap\Omega)\leqslant 2^ne^{n-1}.
\end{equation}
For a given integer $k>0$ let $\rho(k)$ be the supremum of all $r>0$ such that there exists $k$ points $x_1,\ldots, x_k\in\bar\Omega$ with $\dist(x_i,x_j)>r$ for all $i\ne j$. We consider the following cases.

\smallskip
\noindent
{\it Case 1: $\rho(k)\geqslant 1/\sqrt{\kappa}$.} Then for any $r<1/\sqrt{\kappa}$ there exist points $x_1,\ldots,x_k\in\bar\Omega$ such that the balls $B(x_i,r/2)$ are disjoint. Consider the plateau functions $u_i$ supported in $B(x_i,r/2)$ such that $u_i\equiv 1$ on $B(x_i,r/4)$ and $\abs{\nabla u_i}\leqslant 4/r$. Their restrictions to $\Omega$ can be used as test-functions for the Neumann eigenvalue $\lambda_k(g)$, and by the variational principle we obtain
\begin{multline*}
\lambda_k(g)\leqslant\max_i\left(\int_\Omega\abs{\nabla u_i}^2d\vol\right)/\left(\int_\Omega u_i^2d\vol\right)\\
\leqslant 16r^{-2}\max_i\vol(B(x_i,r/2)\cap\Omega)/\vol(B(x_i,r/4)\cap\Omega)\leqslant 2^{n+4}e^{n-1}r^{-2}.
\end{multline*}
Taking the limit as $r\to 1/\sqrt{\kappa}$, we see that $\lambda_k(g)\leqslant C_{10}\kappa$.

\smallskip
\noindent
{\it Case 2: $\rho(k)<1/\sqrt{\kappa}$.} Following the argument above, we see that $\lambda_k(g)\leqslant C_{10}r^{-2}$ for any $0<r<\rho(k)$, and tending $r\to\rho(k)$, we obtain that $\lambda_k(g)\leqslant C_{10}\rho(k)^{-2}$. Now we claim that $\rho(k)\geqslant d/k$. Indeed, to see this we note that the closure a convex domain $\Omega$ contains a geodesic arc whose length equals the diameter $d$. Breaking it into sub-arcs of the length $d/k$, we conclude that  $\rho(k)\geqslant (d/k)$, and hence, $\lambda_k(g)\leqslant C_{10}(k/d)^2$. Taking into account both cases we finish the proof of the theorem.
\qed

\subsection{Proof of Theorem \ref{ueb:dp}: Cheng's inequalities for the Dirichlet problem}
Below we give an argument based on Cheng's comparison theorem for the principal  Dirichlet eigenvalue~\cite{Che75}; however, one can also argue as in the proof of Theorem~\ref{ueb:np} using the volume comparison theorem and constructing test-functions explicitly. 

Denote by $\overline{\dist}(x,y)$ the intrinsic distance on $\Omega$, that is the infimum of lengths of paths in $\Omega$ joining the points $x$ and $y$. Note that the closure of $\Omega$ contains a continuous path $\gamma$ whose length equals the intrinsic diameter $\bar d=\bar d(\Omega)$; its existence follows from the Arzela-Ascoli theorem, see~\cite{BBI} for details. Breaking it into sub-arcs of the length $\bar d/(k+1)$, we find $(k+1)$ points $x_i$ on $\gamma$, where $i=0,\ldots, k$, such that
$$
\overline{\dist}(x_i,x_j)\geqslant2r:=\overline{d}/(k+1)\qquad\text{for any}\quad i\ne j.
$$
In particular, we see that the sets $D(x_i,r)=\{y\in \Omega: \overline{\dist}(x_i,y)< r\}$ are disjoint. Since the extrinsic distance is not greater than the intrinsic distance, we also conclude that each $D(x_i,r)$ lies in $B(x_i,r)\cap\Omega$. 
Denote by $k_0$ the integer $\lfloor\bar d/(4\rad(\Omega))\rfloor $, the greatest integer that is at most $\bar d/(4\rad(\Omega))$. Following the argument in the proof of Lemma~\ref{np:l}, it is straightforward to see that for any $k\geqslant k_0$ and any $0\leqslant i\leqslant k$ there exists $\tilde x_i\in B(x_i,r)\cap\Omega$ such that
\begin{equation}
\label{dp:aux3}
\overline\dist(\tilde x_i,x_i)=\dist(\tilde x_i,x_i)\leqslant r/2\quad\text{and}\quad B(\tilde x_i,r/2)\subset B(x_i,r)\cap\Omega.
\end{equation}
Here in the first relation we used the fact that the point $\tilde x_i$ can be chosen such that $\tilde x_i$ and $x_i$ lie on a extrinsically shortest geodesic arc that is contained in $\Omega$. Since any extrinsically shortest path joining points in the ball $B(\tilde x_i,r/4)$ lies in the ball $B(\tilde x_i,r/2)\subset\Omega$, we conclude that the intrinsic and extrinsic distances coincide on $B(\tilde x_i,r/4)$. Using the first relation in~\eqref{dp:aux3}, it is then straightforward to see that the ball $B(\tilde x_i,r/4)$ lies in $D(x_i,r)$. In particular, the balls $B(\tilde x_i,r/4)$ are disjoint, and by the domain monotonicity principle and Cheng's comparison for the principal eigenvalue, we obtain
$$
\nu_k(\Omega)\leqslant\max_i\nu_0(B(\tilde x_i,r/4))\leqslant\nu_0(B_\kappa(r/4)),
$$
where $B_\kappa(r/4)$ is a ball of radius $r/4$ in the simply connected space of constant sectional curvature $(-\kappa)$. As is shown by Cheng~\cite{Che75}, there is a constant $c_7$ depending on the dimension only such that
$$
\nu_0(B_\kappa(r/4))\leqslant c_7(\kappa+r^{-2}).
$$
 From the consideration above we conclude that for any integer $k\geqslant k_0$ the Dirichlet eigenvalue $\nu_k(\Omega)$ satisfies the inequality
\begin{equation}
\label{aux:k}
\nu_k(\Omega)\leqslant c_7\kappa+c_8((k+1)/\overline{d})^2.
\end{equation}
If $k_0=0$, then the statement is proved. If $k_0\geqslant 1$, then we can estimate the eigenvalue $\nu_{k_0}(\Omega)$ in the following fashion
\begin{equation}
\label{aux:k_0}
\nu_{k_0}(\Omega)\leqslant c_7(\kappa+r_0^{-2})\leqslant c_7(\kappa+\rad^{-2}),
\end{equation}
where we used that 
$$
2r_0:=\bar d/(k_0+1)\geqslant 2\rad(\Omega).
$$
Finally, combining relations~\eqref{aux:k} and~\eqref{aux:k_0}, for any $k\geqslant 0$ we obtain
$$
\nu_k(\Omega)\leqslant\max\{\nu_{k_0}(\Omega),\nu_k(\Omega)\}\leqslant C_{11}(\kappa+\rad^{-2})+C_{12}((k+1)/\overline{d})^2,
$$
which is the desired inequality~\eqref{cheng:dp}.
\qed

\subsection{Proof of Theorem \ref{ueb:dp}: Buser's inequalities for the Dirichlet problem}
We start with recalling that by the Bishop volume comparison theorem for any $0<r\leqslant 1/\sqrt{\kappa}$ the volume of a metric ball $B(x,r)$ in $M$ satisfies the inequality
\begin{equation}
\label{bishop}
\vol(B(x,r))\leqslant n\omega_n\int_0^r t^{n-1}e^{(n-1)t\sqrt{\kappa}}dt\leqslant \omega_ne^{n-1}r^n,
\end{equation}
where $\omega_n$ is the volume of a unit ball in the Euclidean space $\mathbb R^n$, see~\cite{Cha}.

For a given integer $k>0$ let $\rho(k+1)$ be the supremum of all $r>0$ such that there exist $(k+1)$ points $x_0,\ldots,x_k\in\Omega$ with $\dist(x_i,x_j)>r$ for any $i\ne j$. Following the argument in the proof of Theorem~\ref{ueb:np}, we consider the two cases below.

\smallskip
\noindent
{\it Case~1: $\rho(k+1)\geqslant 1/\sqrt{\kappa}$.} For every $r<1/\sqrt{\kappa}$ there exist points $x_0,\ldots, x_k$ such that the balls $B(x_i,r/2)$ are disjoint. When $r\leqslant \rad$, then repeating the argument in the proof of Lemma~\ref{np:l}, we find points $\tilde x_i\in B(x_i,r/2)$ such that
$$
B(\tilde x_i,r/4)\subset B(x_i,r/2)\cap\Omega.
$$
Now by the domain monotonicity and Cheng's comparison for the zero Dirichlet eigenvalue, we have 
\begin{equation}
\label{dp:aux}
\nu_k(\Omega)\leqslant\max_i\nu_0(B(\tilde x_i,r/4))\leqslant c_7(\kappa+r^{-2}).
\end{equation}
Taking the limit as $r\to\min\{\rad,1/\sqrt{\kappa}\}$, we obtain that $\nu_k(\Omega)$ is not greater than $c_9(\kappa+\rad^{-2})$.

\smallskip
\noindent
{\it Case~2: $\rho(k+1)<1/\sqrt{\kappa}$.} Following the line of an argument in Case~1, we see that for any $r<\min\{\rad,\rho(k+1)\}$ relation~\eqref{dp:aux} holds. Tending $r\to\min\{\rad,\rho(k+1)\}$, we obtain
\begin{equation}
\label{dp:aux2}
\nu_k(\Omega)\leqslant c_{7}(\kappa+\rad^{-2}+\rho(k+1)^{-2}).
\end{equation}
Now we estimate the value $\rho(k+1)$. For any given $s$ such that $\rho(k+1)<s<1/\sqrt{\kappa}$ let $m$ be the maximal number of points $y_1,\ldots,y_m\in\Omega$ such that $\dist(y_i,y_j)>s$ for any $i\ne j$. In particular, the balls $B(y_i,s)$, where $i=1,\ldots,m$, cover the domain $\Omega$. By the definition of $\rho(k+1)$ we also conclude that $m\leqslant k$. Thus, by inequality~\eqref{bishop}, we obtain
$$
\vol(\Omega)\leqslant\sum\vol(B(y_i,s))\leqslant m\omega_n e^{n-1}s^n\leqslant c_{10}ks^n.
$$
Letting $s$ tend to $\rho(k+1)$, we further obtain
$$
\rho(k+1)^{-2}\leqslant (c_{10})^{2/n}(k/\vol(\Omega))^{2/n}.
$$
Combining the last relation with inequality~\eqref{dp:aux2}, we get
$$
\nu_k(\Omega)\leqslant c_{7}(\kappa+\rad^{-2})+c_{11}(k/\vol(\Omega))^{2/n}.
$$
Taking into account both cases, we finish the proof of the theorem.
\qed
%
%----------  section ------------------

\subsection*{Acknowledgements}  
The authors are grateful to A. Savo for proposing the idea of  Example~\ref{counter}, as well as to B. Colbois and L. Polterovich for useful discussions. The project has originated out of a number of discussions the authors had while GK and AH were visiting the Centre de recherche math{\'e}matiques (CRM) in Montreal. During the work on the paper AH has been supported by the CRM-ISM postdoctoral fellowship and the postdoctoral programme at  the Max-Planck Institute for mathematics in Bonn. The support  and hospitality of both institutes is gratefully acknowledged. The research of  IP was partially supported by NSERC, FRQNT and Canada Research Chairs program.

\end{document}